\documentclass[]{article}
\usepackage{graphicx}
\usepackage{multirow}
\usepackage{amsmath,amssymb,amsfonts,tikz-cd,makecell}
\usepackage{amsthm}
\usepackage{mathrsfs}
\usepackage[title]{appendix}
\usepackage{xcolor}
\usepackage{textcomp}
\usepackage{manyfoot}
\usepackage{booktabs}
\usepackage{algorithm}
\usepackage{algorithmicx}
\usepackage{algpseudocode}
\usepackage{listings}
\usepackage{hyperref}
\usepackage{enumerate}

\theoremstyle{definition}
\newtheorem{definition}{Definition}
\newtheorem{theorem}{Theorem}
\newtheorem{lemma}[theorem]{Lemma}
\newtheorem{corollary}[theorem]{Corollary}
\newtheorem{proposition}[theorem]{Proposition}
\newtheorem*{remark}{Remark}
\newtheorem{example}[theorem]{Example}

\begin{document}

\title{A full-rank subgroup of Bloch Groups of CM Fields}
\author{Wenhuan Huang}

\maketitle

\begin{abstract}
This article generalizes the results in \cite{MR4264211}
by developing an algorithm for calculating a full-rank subgroup 
of the free part of the Bloch group of a given CM number field.
\end{abstract}

\section*{Introduction}

Let $F$ be a number field, i.e.,  a finite extension of the
rational number field $\mathbb{Q}$. 
Assume $F$ has $r_1$ embeddings into 
$\mathbb{R}$ and $r_2$ pairs of complex embeddings into $\mathbb{C}$.
Then $r_1+2r_2=[F:\mathbb{Q}]$. Let $O_F$ be the 
ring of integers of $F$, i.e.,  the ring of elements in $F$
that are roots of monic polynomials with integer 
coefficients. It is well-known that the unit group $O_F^*$ of
$O_F$ is a finitely
generated abelian group, and has rank $r_1+r_2-1$.
The first regulator of $F$ is an important invariant of 
the field, the definition of which is based on the values
of generators of the torsion-free part of $O_F^*$.
It is closely related to other invariants of
$F$, e.g. the order of the ideal class group of $F$,
and the value of $\zeta_F^*(0)$, where 
$\zeta_F$ is the Dedekind zeta function of $F$
and $\zeta_F^*(m)$ is the leading coefficient
in the Taylor expansion of $\zeta^*_F$ at $s=m$
for an integer $m\leq0$.

The Bloch group of $F$ is defined completely
algebraically, and has various versions, e.g.
Suslin's Bloch group \cite{MR1092031},
extended Bloch group \cite{Zickert+2015+21+54},
and modified Bloch group \cite{MR4264211}. 
Different versions of the Bloch 
group differ slightly in both their torsion and torsion-free parts. 

For every ring $R$ and natural number $n$,
we have abelian groups $K_n(R)$.
$K_0$ and $K_1$ are easy to describe. 
Milnor's $K_2$ functor
is defined both in terms of the Steinberg group and 
in terms of homology of $E(R)$. 
When $R$ is a field,
we have Suslin's stability theorem
$K_n(R)=H_n(GL(R))\simeq H_n(GL_n(R)),$
where the latter isomorphism is induced by inclusion.

For $K$-groups of number fields and their integer rings
we have more tools and results.
Borel's theorem gives the structure of
torsion-free part of $K_i(O_F)$
for every $i>1$ and number field $F$.
In particular, $K_3(O_F)\simeq K_3(F)$ always has rank $r_2$.
In addition, we have Milnor $K$-group $K_3^M(F)$, which has exponent
at most 2 and order $2^{r_1}$. The group $K_3(F)^{ind}$, called
indecomposable $K_3$, is defined as $K_3(F)/K_3^M(F)$,
and has been studied by Suslin \cite{MR1092031} and
Levine \cite{MR1005161}, Burns et al. \cite{BURNS20121502},
and Zickert \cite{MR4264211}.

The general Bloch groups of a field are defined 
inductively in Chap. 1 of \cite{MR1348706}.
The (second) Bloch group $B(F):=B_2(F)$, is known to 
have a homomorphism from $B(F)$ to $K_3(F)_{tf}^{ind}$,
having finite cokernel, and every embedding $\sigma:F\rightarrow\mathbb{C}$
determines a regulator map $reg_\sigma:B(F)\rightarrow\mathbb{R}$. \cite{BURNS20121502}
Therefore, one might be interested in finding a set of 
explicit elements generating the Bloch group.
For the torsion part, a simple algorithm 
has already been given in \cite{Zickert+2015+21+54}.
In this article, we focus on 
finding generators of the torsion-free parts of 
Bloch groups.

There are many previous studies concerning composite fields.
Browkin and Gangl \cite{MR3126639} introduced Brower-Kuroda relations
for Galois extensions, describing relations among the
Dedekind zeta functions of any number field, its Galois extension
and intermediate fields, and the definition of the 
second regulator $r(F)$ and its relation to values of $\zeta_F(2)$.
H. Zhou \cite{MR3352657} studied the relation between generators of $B(F)$ and Bloch groups of
subfields of $F$, as well as the orders of the corresponding $K_2$-groups,
for biquadratic or $S_3$-extensions.
X. Guo and H. Qin \cite{MR3818288} proved that
if $F$ is an imaginary biquadratic number field
with imaginary quadratic subfields $F_1$ and $F_2$, then 
$B(F)_{tf}/(B(F_1)_{tf}+B(F_2)_{tf})$ is always $1$ or $2$,
and determined the necessary and sufficient condition
that $1$ or $2$ holds.
However, in cyclic extensions,
the Brower-Kuroda relation is trivial,
so subfields provide no direct information.
Thus, we have to find other ways to explore the more general cases.

This article is inspired by \cite{MR4264211}, which
presents an algorithm for calculating a generator of 
the torsion-free part of $B(F)$ for an imaginary quadratic field.
This is based on an analysis of $GL_2(O_F)$-tessellation 
of the cone of positive definite binary Hermitian forms,
and then the $GL_2(O_F)$-tessellation of 
some products of $\mathbb{H}_3$.
First, using Vorono\"i theory of Hermitian forms,
we determine all equivalence classes of perfect binary quadratic forms,
and associated ideal polytopes as well as their stabilizer groups
in $GL_2(O_F)$. 
Then we use a 
"descending weighted simplicial complex"
to calculate a basis of $B(F)\otimes\mathbb{Q}$,
and determine its index.

In this article, we generalize the method by Burns et al. \cite{MR4264211}
and develop an algorithm to find a set of elements generating a full-rank
subgroup of $B(F)\otimes\mathbb{Q}$, and determine its index.
We use $(\mathbb{H}_3)^r$ instead of $\mathbb{H}_3$ in general
so that $PGL_2(O_F)$ acts properly discontinuously and almost freely 
on it,
and we can apply Borel's volume formula \cite{MR616899}.
We study the 
homology of 
the Borel-Serre compactification of the 
quotient orbifold $PGL_2(O_F)\backslash(\mathbb{H}_3)^r$,
To avoid possible trouble brought by fundamental units,
we use Goncharov's $cr_2$ in \cite{MR1348706}, instead of that in \cite{MR4264211}.

This article is organized as follows.
In Section 1, we recall some concepts and facts in $K$-theory,
including $K$-groups of rings, various versions of Bloch groups,
the regulator maps connecting them, 
and the role that Dedekind zeta function plays in $K$-theory.
In Sections 2 and 3, we introduce Vorono\"i reduction theory,
consisting of perfect Hermitian forms and equivalence,
as well as its descen to the quotient space of some products
of hyperbolic 3-spaces by the action of $PGL_2(O)$,
and define a "Vorono\"i complex" which makes the calculation more
practical. We recall the algorithm
to compute $GL_n(\mathbb{Z})$-equivalent
classes of rational quadratic forms, as well as 
$T$-perfect forms. Under the condition that $F$ is a totally real
or CM field, we can calculate $GL_n(O_F)$-equivalence classes 
of perfect forms over $F$, which correspond to top dimensional
cells of orbifold $PGL_2(O_F)$.
In Section 4, we give a formula with the data obtained in Section 3
to compute a basis generating a full-rank subgroup
of $B(F)\otimes\mathbb{Q}$.
If $F$ is abelian, then we also determine the index of the basis.

Throughout this article, unless otherwise stated,
we use $\mathbb{Q}$, $\mathbb{Z}$,
$\mathbb{R}$, $\mathbb{C}$
to denote 
the rational number field,
the ring of rational integers,
the real number field, and the complex
number field, respectively.
To avoid misunderstanding,
we use $\mathbb{Z}/n$, rather than $\mathbb{Z}_n$,
to denote $\mathbb{Z}/n\mathbb{Z}$ when $n$ is a positive integer.

For any ring $R$, denote by $R^*$ the 
group of invertible elements of $R$. For any algebraic
number field (a finite extension over $\mathbb{Q}$) $F$,
denote by $O_F$ the ring of integers
in $F$, i.e., the ring of numbers in $F$
having monic minimal polynomials with coefficients
in $\mathbb{Z}$.
For any vector or matrix
$x$, denote by $x^T$ its transpose, and $x^*=\bar{x}^T$ its complex
conjugate transpose. 

For any subfield $F$ of $\mathbb{C}$,
denote by $\mu_F$ the group of cyclic units in $F$. 
We also use $\mu_F$ to denote
a generator of it.

For any finite set (possibly a group) $S$, denote by $|S|$ the cardinality of it.
For any finitely generated abelian group
$G$ and a full-rank subgroup $H$, we say the index of $H$
is $|G/H|$. More generally, if $T$ generates $G/G_{tor}$ 
and $S$ is a basis of $G\otimes\mathbb{Q}$
with $S=PT$, $P\in GL(\mathbb{Q})$,
we say $S$ is of index $|\det(P)|$.

We denote by $\mathbb{H}_3$ the $3$-dimensional hyperbolic space.
Unless otherwise stated, we use the upper-half space model.

\section{Bloch groups and $K_3$}

In this chapter, we introduce various versions of Bloch groups, 
relations among them as well as $K_3$ groups.

\subsection{Bloch groups of fields}

For any (multiplicative) abelian group $A$,
let $\wedge^2A$ denote the quotient of the group
$A\otimes A$ modulo the subgroup generated by 
$a\otimes b+b\otimes a$.
Let $\mathbb{Z}[\mathcal{S}]$ denote the
free abelian group generated by symbols
$[x]$ for $x\in\mathcal{S}$, for any set $\mathcal{S}$.
Let $F^\flat$ denote $F\backslash\{0,1\}$ for any field $F$.

\begin{definition}[\cite{MR1092031}]
  Let $F$ be an arbitrary field having at least 4 elements.
  The pre-Bloch group $\mathcal{P}(F)$ is the group
  $\mathbb{Z}[F^\flat]$ modulo the relation
  \begin{equation}\label{five-term}
    [x]-[y]+[y/x]-[\frac{1-x^{-1}}{1-y^{-1}}]+[\frac{1-x}{1-y}]=0,\forall x\neq y\in F^\flat.
  \end{equation}
  There is a canonical map 
  $\mathcal{P}(F)\rightarrow\wedge^2F^*$
  sending $[x]$ to $x\wedge(1-x)$. The kernel of this map
  $B(F)$ is called Bloch group.

\end{definition}

\begin{remark}
  The group $\mathcal{P}(F)$ is closely related to 
  the scissors congruence group for ideal polyhedra
  (polyhedra whose vertices are some cusps in 
  $\partial\mathbb{H}_3=P_\mathbb{C}^1$)
  in $\mathbb{H}_3$. 
  An oriented ideal tetrahedron with cuspidal
  vertices $[v_0,v_1,v_2,v_3]$ has cross-ratio
  \begin{equation}
    cr_3(v_0,v_1,v_2,v_3):=\frac{\det(v_0,v_3)\det(v_1,v_2)}{\det(v_0,v_2)\det(v_1,v_3)}
  \end{equation}
  and volume $D(cr_3(v_0,v_1,v_2,v_3))$ \cite{MR1937957}, where
  the Bloch-Wigner dilogarithm function is defined as \cite{MR1760901}
  \begin{equation}
    D(z):=\int_{0}^z\log|w|\cdot\arg(1-w)-\log|1-w|\cdot d\arg(w)
  \end{equation}
  The function $D(z)$ has the following properties:
  \begin{equation}
    D(x)-D(y)+D(y/x)-D(\frac{1-x^{-1}}{1-y^{-1}})+D(\frac{1-x}{1-y})=0,\forall x\neq y\in\mathbb{C}^\flat,
  \end{equation}
  \begin{equation}
    D(x)=-D(1-x)=-D(x^{-1})=-D(\bar{x}),\forall x\in\mathbb{C^\flat}.
  \end{equation}
  All of these properties can be obtained from geometry of ideal tetrahedra.
\end{remark}

Next we state special elements
$c_F:=[x]+[1-x]$ and $\langle x\rangle:=[x]+[x^{-1}]$
in $B(F)$,
which will play 
an important role in describing relations among Bloch groups:

\begin{lemma}[Lemma 1.5 in \cite{MR1092031}]
  Let $|F|\geq4$. Then 
  \begin{enumerate}[(1)]
    \item $c_F$ is independent of $x\in F^\flat$;
    \item For any $x\in F^\flat$, $2\langle x\rangle=0$;
    \item $3c_F=2[-1]$, hence $6c_F=0$.
    \item $3c_F=0$ if and only if $char(F)=2$ or 
    $x^2+1=0$ is solvable in $F$, and $2c=0$ if and only if 
    $x^3+1=0$ is solvable in $F$.
  \end{enumerate}
\end{lemma}

Zickert's extended Bloch group
of a number field
\cite{Zickert+2015+21+54}
is isomorphic to $K_3^{ind}(F)$.
However, its definition is quite complicated,
and it is not of our main concern.

The definition of modified Bloch group (see Chap. 3 of \cite{MR4264211}) is as follows.
For any infinite field $F$,
the group $\tilde{\wedge}^2F^*$ is the quotient of the group
$F^*\otimes F^*$ modulo the subgroup generated by 
$x\otimes(-x),x\in F^*$.
Since $x\otimes(-x)=0,\forall x\in F^*$ implies that
\begin{equation}
  0=ab\otimes(-ab)=a\otimes(-a)+a\otimes b+b\otimes a+b\otimes(-b)=a\otimes b+b\otimes a,
\end{equation}
We have that $\wedge^2F^*$
is a quotient of $\tilde{\wedge}^2F^*$.
For $\nu$ a subgroup of $F^*$,
define the homomorphism
\begin{equation}
  \delta_{2,F}^\nu:\mathbb{Z}[F^\flat]\rightarrow\tilde{\wedge}^2F^*/\nu\tilde\wedge F^*
\end{equation}
by sending $[x]$ to the class containing $(1-x)\tilde{\wedge}x$.

Define $\mathfrak{p}(F)$ the group $\mathbb{Z}[F^\flat]$
modulo the relation (\ref{five-term}) and 
\begin{equation}
  [x]+[x^{-1}]=0,[y]+[1-y]=0.
\end{equation}
Then the map $\delta_{2,F}^\nu$ induces a map
\begin{equation}
  \partial_{2,F}^v:\mathcal{P}(F)\rightarrow\tilde{\wedge}^2F^*/\nu\tilde\wedge F^*.
\end{equation}
Define the modified Bloch group
\begin{equation}
  \overline{B}(F)_v:=\ker(\partial_{2,F}^\nu).
\end{equation}

Next we describe the differences among these versions of Bloch groups,
and how to determine their torsion parts:

\begin{theorem}(\cite{Zickert+2015+21+54}\cite{MR4264211})
  Let $F$ be a number field.
  \begin{enumerate}[(1)]
    \item There is an exact sequence
    \begin{equation}
      0\rightarrow\widetilde{\mu_F}\rightarrow\widehat{B}(F)\rightarrow B(F)\rightarrow{0}
    \end{equation}
    where $\widetilde{\mu_F}$ is the unique nontrivial
    $\mathbb{Z}/2$-extension of $\mu_F$.
    \item $\overline{B}(F)$ is exactly $B(F)/\langle c_F\rangle$.
  \end{enumerate}
\end{theorem}

\begin{theorem}[Chap. 8 of \cite{Zickert+2015+21+54}]
  Let $F$ be a number field.
  For every prime number $p$,
  $\nu_p=\max\{\nu|\mu_{p^\nu}+\mu_{p^\nu}^{-1}\in F\}$,
  and $x=\mu_{p^{\nu_p}}$.
  \begin{enumerate}[(1)]
    \item If $p$ is odd, then 
    $\widehat{B}(F)_p$ is cyclic of order $p^{\nu_p}$,
    generated by 
    $$\sum_{i=1}^{p^{\nu_p}}[\frac{(x^{k+1}+x^{-k-1})(x^{k-1}+x^{-k-1})}{(x^k+x^{-k})^2}].$$

    \item 
    $\widehat{B}(F)_2$ is cyclic of order $2^{1+\nu_2}$,
    generated by 
    $$\sum_{i=1}^{2^{\nu_2-1}}[\frac{(x^{k+1}-x^{-k})(x^{k-1}-x^{-k+2})}{(x^k-x^{-k+1})^2}].$$
  \end{enumerate}
\end{theorem}

\begin{corollary}[Sec. 3.4 of \cite{MR4264211}]
  Take notations in the last theorem, and let $d_p$ be the
  largest number $d$ such that $\mu_{p^d}\in F$.
  \begin{enumerate}[(1)]
    \item If $\sqrt{-1}\in F$, then both $|B(F)_2|$ and $|\overline{B}(F)_2|$
    are trivial. Otherwise, $|B(F)_2|=2|\overline{B}(F)_2|={2^{d_2-1}}$.
    \item If $\mu_3\in F$, then both $|B(F)_3|$ and $|\overline{B}(F)_3|$
    are trivial. Otherwise, $|B(F)_3|=3|\overline{B}(F)_3|={3^{d_3}}$.
    \item If $\mu_p\in F$, then both $|B(F)_p|$ and $|\overline{B}(F)_p|$
    are trivial. Otherwise, $|B(F)_p|=|\overline{B}(F)_p|=p^{d_p}$.
  \end{enumerate}
\end{corollary}

\subsection{$K_3$ and Bloch groups}

In this section, we introduce the group $K_3$
of number fields and their integer rings, as well as
their connection to Bloch groups.

\begin{definition}[Sec. 5.1 and 5.2 of \cite{MR1282290}]
  \begin{enumerate}[(1)]
  \item Let $G$ be a group, and $X$ a contractible
  $G$-space (topological spaces equipped with homeomorphic
  $G$-actions), and $G$ acts freely and properly
  discontinously on $X$, and $G\backslash X$ is paracompact.
  Up to homotopy equivalence, $G\backslash X$ is unique.
  We call $BG:=G\backslash X$ a classifying space of $G$.
  
  \item Let $R$ be a ring,
  and $BGL(R)^+$ be Quillen's plus construction of 
  $BGL(R)$, and $K_0(R)$ is given discrete topology.
  The $i$-th $K$-group
  $K_i(R)$ is defined as the $i$-th
  homotopy group $\pi_i(BGL(R)^+\times K_0(R))$.
  \end{enumerate}
\end{definition}

Next we introduce some known properties of 
$K$-groups of number fields and their integer rings
that may be useful.

\begin{theorem}[Pages 139-146 of \cite{MR2181823}]
  For every number field $F$,
  \begin{enumerate}[(1)]
    \item $K_0(F)\simeq\mathbb{Z}$, $K_0(O_F)\simeq\mathbb{Z}\oplus Cl(F)$;
    \item $K_1(F)\simeq F^*$, $K_1(O_F)\simeq O_F^*$;
    \item $K_2(O_F)$ is finite abelian, and is 
    the kernel of the map
    \begin{equation}
      K_2(F)\rightarrow\oplus_{v}F_v^*,
      \{a,b\}\mapsto(-1)^{v(a)v(b)}a^{v(b)}b^{-v(a)}\pmod v
    \end{equation}
    where $\{*,*\}$ is the Steinberg symbol,
    and $v$ runs over all finite places of $F$.
    \item If $i\geq3$, then the rank of $K_i(O_F)$ is
    $r$ if $i\equiv3\pmod4$, $r_1+r_2$ if $i\equiv1\pmod4$,
    and $0$ if $i$ is even. 
  \end{enumerate}
\end{theorem}

\begin{remark}
  Milnor's $K$-groups for $n\geq3$ \cite{MR260844}, is defined as 
  \begin{equation}
    K_n^M(F):=(F^*)^{\otimes n}/\langle Q\rangle
  \end{equation}
  where $Q$ is the subgroup generated by $q_1\otimes q_2\otimes\dots\otimes q_n$
  for some $q_i+q_{i+1}=1$.
  It is isomorphic to $(\mathbb{Z}/2)^{r_1}$. \cite{MR442061}
  In particular, $K_3^M(F)$ injects into $K_3(F)$,
  and the cokernel is denoted $K_3^{ind}(F)$,
  which is isomorphic to $\widehat{B}(F)$. See Sec. 2.3 of \cite{Zickert+2015+21+54}.
\end{remark}

For higher $K$-groups we have the following general
relations between $K_i(F)$ and $K_i(O_F)$:

\begin{theorem}[Page 146 of \cite{MR2181823}]
  Let $F$ be a number field and $i>1$.
  \begin{enumerate}[(a)]
    \item If $i$ is odd then $K_i(O_F)\simeq K_i(F)$;
    \item If $i$ is even then $K_i(O_F)$ is finite,
    but $K_i(F)$ is an infinite torsion group,
    satisfying the exact sequence
    \begin{equation}
      0\rightarrow K_i(O_F)\rightarrow K_i(F)\rightarrow\oplus_{\mathcal{P}\ prime} K_{i-1}(O_F/\mathcal{P})\rightarrow0.
    \end{equation}
  \end{enumerate}
\end{theorem}

\begin{theorem}[Page 148 in \cite{MR2181823}]
  If $i>1$, then the $K$-group of the finite field
  \begin{equation}
    K_{2i-1}(\mathbb{F}_q)\simeq \mathbb{Z}/(q^i-1), K_{2i}(\mathbb{F}_q)=0.
  \end{equation}
\end{theorem}

\subsection{Regulator maps and Dedekind zeta functions}

Borel \cite{MR506168} and Beilinson \cite{MR760999}
defined regulator maps 
$K_{2m-1}(\mathbb{C})\rightarrow\mathbb{R}(m-1):=(2\pi i)^{m-1}\mathbb{R}$,
and Beilinson's is exactly twice Borel's \cite{MR1869655}.
We denote by $reg_m$ the Beilinson's regulator map.
It is compatible with the natural actions of complex conjugation
on $K_{2m-1}(\mathbb{C})$ and $\mathbb{R}^{m-1}$.
For each embedding $\sigma:F\rightarrow\mathbb{C}$, we 
consider the composite homomorphism
\begin{equation}
  reg_{m,\sigma}:K_{2m-1}(F)\stackrel{\sigma_*}{\longrightarrow}K_{2m-1}(\mathbb{C})\stackrel{reg_m}{\longrightarrow}\mathbb{R}(m-1),
\end{equation}
where $\sigma_*$ denotes the map on $K$-groups
induced by inclusion $\sigma:F\rightarrow\mathbb{C}$.

We recall some concepts and facts about Dedekind zeta functions \cite{MR3126639}.
Let $F$ be a number field.
For $Re(s)>1$, the Dedekind zeta function
\begin{equation}
  \zeta_F(s):=\sum_{\mathfrak{a}}(N_{F/\mathbb{Q}}(\mathfrak{a}))^{-s}
\end{equation}
where $\mathfrak{a}$ runs over all non-zero integral ideals of $O_F$.
It has an Euler product representation
\begin{equation}
  \zeta_F(s)=\prod_{\mathcal{P}}(1-(N_{F/\mathbb{Q}}(\mathcal{P}))^{-s})^{-1}
\end{equation}
where $\mathcal{P}$ runs over all prime ideals of $O_F$.

The function $\zeta_F(s)$ converges absolutely for $Re(s)>1$,
and has an analytic continuation to $\mathbb{C}\backslash\{1\}$
as follows. Let 
\begin{equation}
  \xi_F(s):=\left(\frac{|\Delta_F|}{4^{r}\pi^n}\right)^{s/2}\Gamma(s/2)^{r_1}\Gamma(s)^{r_2}\zeta_F(s),
\end{equation}
then $\xi_F$ has two simple poles, 0 and 1. The continuation is based on 
this equality:
\begin{equation}
  \xi_F(s)=\xi_F(1-s).
\end{equation}

The function $\zeta_F(s)$ vanishes at all negative even integers.
Moreover, we have Borel's theorem (See Theorem 2.1 of \cite{MR4264211}):
\begin{theorem}
  Let $F$ be a number field, and $i\geq2$.
  \begin{enumerate}[(a)]
    \item Write $reg_{m,F}$ for the map 
    \begin{equation}
      K_{2m-1}(F)\rightarrow\prod_{\sigma:F\hookrightarrow\mathbb{C}}\mathbb{R}(m-1),
      x\mapsto(reg_{m,\sigma}(x))_\sigma.
    \end{equation}
    Then the image of $reg_{m,F}$ is a full-rank lattice
    in $\mathbb{R}^{m-1}$.
    \item Let $R_m(F)$ be the covolume of the image of $reg_{m,F}$.
    Then $\zeta_F(s)$ vanishes to order $r_2$ (resp. $r_1+r_2$) at 
    $s=1-m$ if $m$ is even (resp. odd).
    \item Let $\zeta_F^*$ denote the leading coefficient in the Taylor
    expansion of $\zeta_F(s)$. Then $\zeta_F^*(1-m)=q_{m,F}R_m(F)$ for some 
    $q_{m,F}$ in $\mathbb{Q}^*$.
  \end{enumerate}
\end{theorem}
\begin{remark}
  We call $R_i(F)$ the $i$-th regulator of the field $F$.
\end{remark}

One might be interested in the exact value of $q_{m,F}$.
Lichtenbaum \cite{MR406981} conjectured
\begin{equation}
  |q_{m,F}|=2^{n_{m,F}}\frac{|K_{2m-2}(O_F)|}{|K_{2m-1}(O_F)_{tor}|}.
\end{equation}
Bloch and Kato \cite{MR1086888} put forward a 
Tamagawa number conjecture. In some special cases, we already have
(see Corollary 2.6 of \cite{MR4264211})
\begin{theorem}\label{BK1}
  Bloch-Kato conjecture for $h^0(\mathrm{Spec}(F))(1-m)$ is valid if 
  $F$ is an abelian field.
\end{theorem}
On the other hand, the validity of Bloch-Kato conjecture
is ralated to the value of $q_{m,F}$ in Theorem 1.13 \cite{MR4264211}.
In this article we focus on the corollary for $m=2$.

\begin{proposition}\label{BK2}
  If Bloch-Kato conjecture is valid for motive
  $h^0(\mathrm{Spec}(F))(-1)$, then
  \begin{equation}
    q_{2,F}=(-2)^{r_1+r_2}\frac{|K_2(O_F)|}{|K_3(O_F)_{tor}|}.
  \end{equation}
\end{proposition}

\begin{example}
  \begin{enumerate}[(a)]
    \item If $F$ is an imaginary quadratic field, Then
    $|K_3(O_F)|$ is always 24,
    and $\zeta_F^*(-1)=-\frac{1}{12}|K_2(O_F)|R_2(F)$.
    \item Assume that $F$ is an imaginary cyclic quartic field.
    Then $|K_3(O_F)|$ is 48 if $\sqrt{2}\in F$,
    120 if $\sqrt{5}\in F$, and 24 otherwise.
    Therefore, $\frac{\zeta_F^*(-1)}{|K_2(O_F)|R_2(F)}=\frac{1}{6}$,
    $\frac{1}{12}$ and $\frac{1}{30}$ in these cases respectively.
  \end{enumerate}
\end{example}

\begin{remark}
  The calculation of $|K_2(O_F)|$ is of great interest
  of $K$-theory researchers. 
  Belabas and Gangl \cite{MR2067570}
  collected theoretical methods and results of some
  quadratic and miscellaneous fields.
  As for the computational method, Tate \cite{MR429836} introduced
  a lemma, later applied by L. Zhang and K. Xu \cite{MR3454374}
  \cite{MR3958056} to prove the triviality for some 
  imaginary cyclic quartic cases.
\end{remark}

\section{Vorono\"i theory of positive definite quadratic forms}

\subsection{Vorono\"i reduction theory}

We state some concepts and facts about perfectness of 
quadratic forms, and geometry related to them. 

Let $Her_m(\mathbb{C})$ (resp. $Sym_m(\mathbb{R})$) denote the set of $m\times m$ complex Hermitian 
(resp. $m\times m$ real symmetric) matrices,
viewed as $m^2$-dimensional (resp. $m(m+1)/2$-dimensional) real vector space.
Let $F$ be a number field having $r_1$ real places and $r_2$ complex places.
By taking non-conjugate embeddings of $F$ into $\mathbb{C}$,
we can view $Her_m(F)$ as a subset of 
\begin{equation}
  S_{m,F}:=(Sym_{m}(\mathbb{R}))^{r_1}\times(Her_m(\mathbb{C}))^{r_2},
\end{equation}
which allows us to view $Her_m(F)$ as a $\mathbb{Q}$-vector space with
dimension $r_1\frac{m(m+1)}{2}+r_2m^2$, and the rational points of $S_{r_1,r_2}$
are exactly $Her_m(F)$.

Define a map $q:O_F^m\rightarrow Her_m(F)$ by the outer product
\begin{equation}
  q(x)=xx^*,
\end{equation}
where $x^*:=\bar{x}^T$. Every $A\in Her_m(F)$ defines an Hermitian form by
\begin{equation}
  A[x]:=Tr_{F/\mathbb{Q}}(x^*Ax), \forall x\in F^m.
\end{equation}
Define the bilinear map
\begin{equation}
  \langle*,*\rangle: S_{m,F}\times S_{m,F}\rightarrow \mathbb{C}, \langle A,B\rangle\mapsto Tr(AB).
\end{equation}
Then for $v\in O_F^m$ (identified with its image in $(\mathbb{R}^m)^{r_1}\times(\mathbb{C}^m)^{r}$), we have
\begin{equation}
  A[x]=Tr_{F/\mathbb{Q}}(\langle A,q(x)\rangle),
\end{equation}
which always take values in $\mathbb{Q}$.

Let $Her_m^{>0}(\mathbb{C})$ (resp. $Sym_m^{>0}(\mathbb{R})$) denote the cone of 
$m\times m$ positive definite Hermitian matrices
(resp. real symmetric matrices),
and $Her_m^{>0}(F)$ the set of totally positive definite Hermitian matrices over $F$,
viewed as a subset of 
\begin{equation}
  S_{m,F}^{>0}:=(Sym_{m}^{>0}(\mathbb{R}))^{r_1}\times(Her_m^{>0}(\mathbb{C}))^{r}.
\end{equation}

\begin{definition}
  \begin{enumerate}[(a)]
    \item For $A\in Her_m^{>0}(F)$, define the minimum of $A$ 
    \begin{equation}
      \mathrm{min}(A):=\displaystyle\inf_{x\in O_F^m\backslash\{0\}}A[x],
    \end{equation}
    and the set of minimum points 
    \begin{equation}
      \mathrm{Min}(A):=\{x\in O_F^m|A[x]=\mathrm{min}(A)\}.
    \end{equation}
    \item An Hermitian form $A\in Her_m^{>0}(F)$ is called perfect, if
    \begin{equation}
      \mathrm{span}_\mathbb{R}\{q(v)|v\in\mathrm{Min}(A)\}=S_{m,F}.
    \end{equation}
    \item Two Hermitian forms $A,A'\in Her_m^{>0}(F)$ are called
    equivalent, if
    there exists a matrix 
    $P$ in $GL_m(O_F)$, such that 
    \begin{equation}
      \langle A,xx^*\rangle=\langle P^*A'P,xx^*\rangle
    \end{equation}
    for every $x\in F^m$.
  \end{enumerate}
\end{definition}

The set $q(O_F^m)$ is discrete in $S_{m,F}$ and the level set
$\{q(x)|\langle A,q[x]\rangle=C\}$ is compact for every $C>0$.
So $\mathrm{Min}(A)$ is always finite
if $A$ is totally positive definite. Besides, every vector 
$v\in O_F^m$ determines a linear functional $q(v)$ on $Her_m(F)$,
via $A\mapsto\langle A,q(v)\rangle$.

Regarding perfectness, we have the following theorems,
for the version of perfect lattices 
in Theorems 3.1 and 3.2 of \cite{MR2769341},
and for the version of Hermitian forms 
in Proposition 6 of \cite{MR1681629}:

\begin{theorem}
  \begin{enumerate}[(a)]
    \item An Hermitian form over $F$ is perfect, if and only if it is uniquely
    determined by $\mathrm{min}(A)$ and $\mathrm{Min}(A)$.
    \item Fix a totally real or CM field $F$ and dimension $m$. 
    If $A\in S_{m,F}$ is perfect with minimum 1,
    then $A$ is $F$-rational, i.e.,  $A\in Her_m(F)$. Moreover, 
    up to $GL_2(O_F)$-equivalence and $F$-homotheties,
    there are only finitely many $m$-ary perfect Hermitian forms over $F$.
  \end{enumerate}
\end{theorem}

\begin{remark}
    \begin{enumerate}[(a)]
    \item 
    The perfect forms over general fields are closely related 
    to $T$-perfect forms over $\mathbb{Q}$,
    introduced in Chap. 7 of \cite{MR2537111}.
    To calculate all $GL_2(O_F)$-equivalence classes of $m$-ary perfect forms 
    over $F$ having minimum $1$, one only needs to calculate $mn$-ary rational 
    perfect forms that lie in a $(r_1\frac{m(m+1)}{2}+r_2m^2)$-dimensional subspace
    of $Sym_{(mn)^2}(\mathbb{Q})$, the $\mathbb{Q}$-vector space
    of $mn$-ary symmetric matrices.
    For application to imaginary quadratic fields, see \cite{MR3457984}.
    \item 
    The algorithm for computing rational perfect forms,
    including $T$-perfect forms, are already known.
    For $F=\mathbb{Q}$ and $m\leq8$, all equivalence classes of 
    perfect quadratic fields are comfirmed. When it comes to 
    $m=8$, the complexity of computation becomes very high, 
    making the computation take months. See \cite{MR2466406}, \cite{MR2300003}.
    \end{enumerate}
\end{remark}

\subsection{Hyperbolic 3-space as symmetric space}

From now on, we assume $F$ to be a CM field, i.e.,  an imaginary
quadratic extension over a totally real field,
$r=r_2$,
and we only use $m=2$.
Then the situation is as follows:

Let $G=GL_2(\mathbb{C})$, $K=U_2$ be a maximal compact subgroup of $G$.
Then $X=G/K$ can be viewed as $Her_2^{>0}(C)$.
Modulo homotheties,
let $G'=PGL_2(\mathbb{C})$, $K'=PU_2$ a maximal compact subgroup of $G'$.
Then $X'=G'/K'$ can be viewed as $\mathbb{H}_3$. 
Let $P_\mathbb{C}^1=\partial\mathbb{H}_3$ be the boundary of $\mathbb{H}_3$,
and denote $\overline{\mathbb{H}_3}=\mathbb{H}_3\cup P_\mathbb{C}^1$.
Then $G$ acts on $\overline{\mathbb{H}_3}$ transitively and continuously,
and preserves the measure on $\mathbb{H}_3$.

\section{Computation of Perfect Forms and Cell Complexes}

For the convenience of readers unfamiliar with 
perfect forms, we give a detailed description of 
rational quadratic forms, $T$-perfect forms and perfect forms over 
general number fields.
We also introduce the computational method 
developed by Sch\"urmann \cite{MR2537111}\cite{MR2300003}.

With the data of perfect forms, one can obtain 
a decomposition of 
$Her_2^{>0}(\mathbb{C})^{r}$, 
and thus a 
tessellation of $(\mathbb{H}_3)^r$,
which has ideal faces with $F$-rational vertices.

\subsection{Rational quadratic forms}

In this section, let $F=\mathbb{Q}$.
Let $A$ be an $m\times m$ positive definite matrix
over $\mathbb{Q}$. Then $A[x]$ is
the quadratic form with matrix $A$.
For every $\lambda>0$, the set 
\begin{equation}
  E(A,\lambda):=\{x\in\mathbb{R}^d|A[x]\leq\lambda\}
\end{equation}
is a closed, strictly convex subset of $\mathbb{R}^m$
with center 0. By Minkowski's convex body theorem, 
if $E(A,\lambda)$ has volume at least $2^d$, then 
$E(A,\lambda)$ contains a point in $\mathbb{Z}^d\backslash\{0\}$.
So there is a smallest $\lambda(Q)$ such that 
$E(A,\lambda)$ contains a non-zero integral point.
It is actually the minimum of $A$.
Moreover, Hermite \cite{MR1578698} proved 
\begin{equation}
  \mathrm{min}(A)\leq(\det(A))^{1/m}\left(\frac{4}{3}\right)^{(m-1)/2}.
\end{equation}
We call
\begin{equation}
  \mathcal{H}_d:=\sup_{A\in Sym_d^{>0}(\mathbb{Q})}\frac{\mathrm{min}(Q)}{\det(Q)^{1/m}}
\end{equation}
the Hermite constant, which depends only on $m$.
Moreover, define the Ryshkov polyhedron \cite{MR276873}
\begin{equation}
  \mathcal{R}:=\{A\in Sym_m^{>0}(\mathbb{Q})|\mathrm{min}(A)\geq1\}.
\end{equation} 
Then we have 
\begin{equation}
  \mathcal{H}_m=1/\inf_\mathcal{R}(\det(A))^{1/m}.
\end{equation}
Then we have the fundamental identity
$A[x]=\langle A,xx^*\rangle$. So we see the 
Ryshkov polyhedra are intersections of infinitely many halfspaces:
\begin{equation}
  \mathcal{R}=\{A\in Sym_m^{>0}(\mathbb{Q})|\langle A,xx^*\rangle\geq1,\forall x\in\mathbb{Z}^m\}.
\end{equation}
It is a locally finite polyhedron, i.e., its 
intersection with any polytope os a polytope.

Vorono\"i \cite{MR1580754} first proved that 
there exist only finitely many $m$-ary rational perfect forms,
given $m$ and the minimum, up to $GL_m(\mathbb{Z})$-equivalence
and scaling.

The vertices of $\mathcal{R}$ are exactly perfect forms.
The Vorono\"i graph in dimension $m$ is as follows:
the vertices are $m$-ary rational perfect forms,
and $AA'$ is connected if the convex hull of
$AA'$ is an edge of $\mathcal{R}$.
In this case, we say $A$ and $A'$ are contiguous perfect forms,
or Vorono\"i neighbors.
The following Vorono\"i's algorithm,
which is a graph traversal alghorithm,
can be used to enumerate perfect forms up to
equivalence and scaling:

\begin{algorithm}
\caption{Voromo\"i's algorithm}
\begin{algorithmic}[1]
\Require Dimension $m$ and an $m$-ary perfect form $A$
\Ensure A complete list of inequivalent $m$-ary perfect forms
\State $i\gets 0, A_0\gets A$
\State $M\gets\mathrm{Min}(A_i)$
\State $\mathcal{P}(A_i)\gets\{A'\in Her_m^{>0}(\mathbb{Q})|A'[x]\geq0,\forall x\in\mathrm{Min}(A_i)\}$
\State Transfer the polyhedron $\mathcal{P}(A_i)$ to
V-description, and enumerate extreme rays $R_1,\dots,R_k$ of the cone 
$\mathcal{P}(A_i)$
\While{$R\in\{R_1,\dots,R_k\}$}
\State Determine contiguous forms $A_i+\alpha R$
\If{$A_i+\alpha R$ is not equivalent to a known form}
    \State Add $A_i+\alpha R$ after the sequence $\{A_n\}$
\EndIf
\EndWhile
\State $i\gets i+1$, go back to step 2, until no new perfect forms added
\end{algorithmic}
\end{algorithm}

As for the initial form, we have for example
\begin{equation}
  A[x]=\sum_{i=1}^m x_i^2-\sum_{i=1}^{m-1}x_ix_{i+1}
\end{equation}
which is called Vorono\"i's first perfect form,
associated to the root lattice \cite{MR1957723}.

As for the computation of the minimum and minimum points of a 
certain positive definite form, we can use the algorithm of 
Fincke and Pohst \cite{MR1228206}:
Given a positive definite form $A$, we only need to compute all 
$x\in\mathbb{Z}^m$ with $A[x]\leq C$ for some constant $C>0$.
Let $C=\min_{1\leq i\leq m}q_{ii}$, then $A[x]\leq C$
has at least one integer solution, hence $\min(A)\leq C$.
Note that it is easy to write 
\begin{equation}
  A[x]=\sum_{i=1}^m A_i\left(x_i-\sum_{j=i+1}^m\alpha_{ij}x_j\right)^2
\end{equation}
with $A_i,\alpha_{ij}\in\mathbb{R}$. Therefore the minimum points must 
satisfy 
\begin{equation}
  \left|x_i-\sum_{j=i+1}^m\alpha_{ij}x_j\right|\leq\sqrt{\frac{C}{A_i}}
\end{equation}
for $i=m,m-1,\dots,1$. Obviously the number of 
integer solutions is finite.

We recall some facts about descriptions of 
polyhedra. As a subset of a finite-dimensional 
Euclidean space $\mathbb{E}$, a (closed) polyhedron $P$ can be 
described by finitely many linear inequalities:
\begin{equation}
  P=\{x\in\mathbb{E}|\mathcal{A}x\geq b\}
\end{equation}
where $\mathcal{A}$ is a matrix with $dim(\mathbb{E})$ columns,
and has same rows as $b$, which is a $\mathbb{R}$-column vector.
It is called H-description of $P$. Besides, 
by Farkas-Minkowski-Weyl Theorem \cite{MR874114}
$P$ can also be described by some finite set of generators:
\begin{equation}
  \begin{split}
    P&=\mathrm{conv}\{v_1,\dots,v_k\}+\mathrm{cone}\{v_{k+1},\dots,v_n\}\\
   &=\{\sum_{i=1}^{n}\lambda_iv_i|\lambda_i\geq0,\sum_{i=1}^{k}\lambda_i=1\}
  \end{split}
\end{equation}
where $v_i\in\mathbb{E}$ for $1\leq i\leq n$.
This is called a V-description of $P$.
If $n$ is minimum, then $v_1,\dots,v_k$ are vertices 
and $v_{k+1},\dots,v_n$ are rays of $P$.
The algorithm for transformation between 
H-description and V-description is known \cite{MR874114}.

In line 6 of Algorithm 1, the contiguous form of $A$
are of the form $A+\rho R$, where $\rho$ is the smallest positive
number such that 
\begin{equation}
  \min(A+\rho R)=\min(A),\mathrm{Min}(A+\rho R)\not\subset\mathrm{Min}(A).
\end{equation}
The following algorithm determines the value of $\rho$:

\begin{algorithm}
\caption{Determination of Vorono\"i neighbors}
\begin{algorithmic}[1]
\Require A perfect form $A$ and an extreme ray $R$
\Ensure $\rho>0$, such that (3.12) holds

\State{$l\gets0,u\gets1$}

\While{$A+uR\notin Sym_m^{>0}(\mathbb{R})$ or $\min(A+uR)=\min(A)$}

\If{$A+uR\notin Sym_m^{>0}(\mathbb{R})$}
\State{$u\gets(l+u)/2$}
\Else
\State{$l\gets u,u\gets 2u$}
\EndIf
\EndWhile

\While{$\mathrm{Min}(A+lR)\subseteq\mathrm{Min}(A)$}

{$\gamma\gets(l+u)/2$}
\If{$\mathrm{Min}(A+uR)\not\subset\mathrm{Min}(A)$ and $\min(A+uR)=\min(A)$}
\State{$l\gets u$}
\ElsIf{$\min(A+\gamma R)\geq\min(A)$}
\State{$l\gets\gamma$}
\Else 
\State{
  $u\gets\min\{(\min(A)-A[v])/R[v]|v\in\mathrm{Min}(A+\gamma R),R[v]<0\}\cup\{\gamma\}$
}
\EndIf

\EndWhile
\State{$\rho\gets l$}

\end{algorithmic}
\end{algorithm}

\begin{remark}
  The Algorithm 2 is slightly different from Algorithm 2
  of Sch\"urmann \cite{MR2537111}. This is because the author
  find it possible to be trapped into infinite loop in 
  the original algorithm if the exact value of 
  $\rho$ satisfying (3.12) is already stored in $u$.
\end{remark}

In the first while loop, the procedure determines a lower bound 
$l$ and upper bound $u$ of $\rho$, such that 
$A+lR$ and $A+uR$ are both positive definite, with
$\min(A+lR)=\min(A)$, $\min(A+uR)<\min(A)$.
In the second while loop, the value of $\rho$ is finally 
determined.

As for line 7 of Algorithm 1, the equivalence test,
note that
if $A$ and $A'$ are both $m$-ary perfect forms,
then $A$ and $A'$ are $GL_m(\mathbb{Z})$-equivalent if and only if 
they have the same minimum and their minimum points sets 
are equivalent. The whole process above can be implemented in 
software. For example, SageMath collects
many kinds of algorithms, including the above algorithms.

\begin{remark}
  The number of equivalence classes of rational perfect forms,
  as well as time complexity of computation,
  increases rapidly when $m$ increases.
  The number of equivalence classes is less than 8
  when $m\leq6$, 33 if $m=7$, and 10916 if $m=8$.
  Moreover, when $m\leq6$, the calculation
  can be done by hand, and the case $m=7$
  takes a little computer time. However
  it takes months when it comes to $m=8$. \cite{MR2300003}
  The author finds that most of computation time
  is consumed by line 4 of Algorithm 1, the 
  transformation from $H$-description to 
  $V$-description of a polyhedral cone.
\end{remark}

\subsection{$T$-perfect forms}

Let $T\subseteq Sym_m(\mathbb{R})$ be a linear subspace.
The intersection 
\begin{equation}
  \mathcal{R}_T:=\mathcal{R}\cap T
\end{equation}
is also a locally finite polyhedron.
Its vertices are called $T$-perfect forms.
Generally $T$-perfect forms are not 
necessarily globally perfect,
and finiteness of $T$-perfect forms up to 
$GL_m(\mathbb{Z})$-equivalence may fail \cite{MR1308129}.
However, we can still generalize Vorono\"i's algorithm
to a graph traversal search of $T$-equivalent
$T$-perfect forms. Here two $T$-perfect forms
are said to be $T$-contiguous if they are connected by 
an edge of $\mathcal{R}_T$.

\begin{remark}
  One might be interested in $G$-perfect forms.
  Let $G$ be a finite subgroup of $GL_m(\mathbb{Z})$.
  The space of $G$-invariant quadratic forms
  \begin{equation}
    T_G:=\{A\in Sym_m(\mathbb{R})|U^*AU=A,\forall U\in G\}
  \end{equation}
is a linear subspace of $Sym_m(\mathbb{R})$.
The $T_G$-perfect forms are also called $G$-perfect forms.
The $G$-equivalence classes of $G$-perfect forms are finite 
\cite{MR1413575}.
\end{remark}

The following algorithm, adapted from 
Vorono\"i's Algorithm 1, can be used to find 
equivalence classes of $T$-perfect forms,
although it may never terminate,
or may reach a "dead end" (All extreme rays of $\mathcal{P}_T(A)$
are positive semidefinite, and therefore 
the ray $\{A+\alpha R|\alpha>0\}$ has no $T$-contiguous
$T$-perfect form of $A$).

\begin{algorithm}
\caption{Voromo\"i's algorithm}
\begin{algorithmic}[1]
\Require Dimension $m$ and a linear subspace $T$ of $Sym_m(\mathbb{R})$,
and a $T$-perfect form $A$
\Ensure A complete list of $T$-inequivalent $m$-ary $T$-perfect forms
\State $i\gets 0, A_0\gets A$
\State $M\gets\mathrm{Min}(A_i)$
\State $\mathcal{P}_T(A_i)\gets\{A'\in Her_m^{>0}(\mathbb{Q})|A'[x]\geq0,\forall x\in\mathrm{Min}(A_i)\}$
\State Transfer the polyhedron $\mathcal{P}_T(A_i)$ to
V-description, and enumerate extreme rays $R_1,\dots,R_k$ of the cone 
$\mathcal{P}_T(A_i)$
\While{$R\in\{R_1,\dots,R_k\}$}
\State Determine $T$-contiguous forms $A_i+\alpha R$
\If{$A_i+\alpha R$ is not $T$-equivalent to a known form}
    \State Add $A_i+\alpha R$ after the sequence $\{A_n\}$
\EndIf
\EndWhile
\State $i\gets i+1$, go back to step 2, until no new perfect forms added
\end{algorithmic}
\end{algorithm}

Another difference is that, compared with theperfect cases, we do not have
an initial $T$-perfect form to start in general.
With some adjustments of Algorithm 2, we have the following algorithm
4 to find an initial $T$-perfect form.

\begin{algorithm}
\caption{Determination of an initial $T$-perfect form}
\begin{algorithmic}[1]
\Require Dimension $m$, a linear subspace $T$ of $Sym_m(\mathbb{R})$,
and a positive definite form $Q$ in $T$
\Ensure A $T$-perfect form $A$ that has the same minimum as $Q$
\State{$A\gets Q_0$}
\State{$L\gets$ the maximal linear subspace of $\mathcal{P}_T(A)$}
\While{$L$ is not trivial}

\State{Choose a form $R\in L\backslash\{0\}$, which is indefinite}

\State{$l\gets0,u\gets1$}

\While{$A+uR\notin Sym_m^{>0}(\mathbb{R})$ or $\min(A+uR)=\min(A)$}

\If{$A+uR\notin Sym_m^{>0}(\mathbb{R})$}
\State{$u\gets(l+u)/2$}
\Else
\State{$l\gets u,u\gets 2u$}
\EndIf
\EndWhile

\While{$\mathrm{Min}(A+lR)\subseteq\mathrm{Min}(A)$}

{$\gamma\gets(l+u)/2$}
\If{$\mathrm{Min}(A+uR)\not\subset\mathrm{Min}(A)$ and $\min(A+uR)=\min(A)$}
\State{$l\gets u$}
\ElsIf{$\min(A+\gamma R)\geq\min(A)$}
\State{$l\gets\gamma$}
\Else 
\State{
  $u\gets\min\{(\min(A)-A[v])/R[v]|v\in\mathrm{Min}(A+\gamma R),R[v]<0\}\cup\{\gamma\}$
}
\EndIf

\EndWhile
\State{$\rho\gets l$}
\State{$A\gets A+\rho R$}
\State{$L\gets$ the maximal linear subspace of $\mathcal{P}_T(A)$}

\EndWhile

\end{algorithmic}
\end{algorithm}

\subsection{Perfect forms over totally real and CM fields}

Let $m>0$ be an integer and $F$ be a totally real field or 
a CM field. By Theorem 2.2(2), we can investigate perfect 
Hermitian forms over $F$ by studying $T$-perfect forms 
as follows.

Let $F$ be totally real with $r_1$ real places, and
has an integral basis $a_1,a_2,\dots,a_{r_1}$.
Let $A$ be a totally positive definite $m\times m$ matrix of $F$,
i.e., $\sigma(A)\in Sym_m^{>0}(\mathbb{R})$ for every
embedding $\sigma:F\hookrightarrow\mathbb{R}$.
$A[x]=Tr_{F/\mathbb{Q}}(A,q(x))$ always takes values in 
$\mathbb{Q}$. If we write $x=(x_1,\dots,x_m)^T$
where 
\begin{equation}
  x_i=\sum_{j=1}^{r_1}x_{i,j}a_j, x_{i,j}\in\mathbb{Z},
\end{equation}
then $A[x]$ can also be viewed as
an $mr_1$-ary rational quadratic form, lying in a subspace
$T$ of $Sym_{mr_1}(\mathbb{R})$ with dimension
$\frac{m(m+1)}{2}r_1$. By Theorem 2.2,
The calculation of $T$-perfect forms terminates in finitely
many steps.

The CM field cases are similar.
Let $F$ be a CM field with $r_2$ complex places, and
has an integral basis $a_1,a_2,\dots,a_{2r_2}$.
Let $A$ be a totally positive definite $m\times m$
Hermitian matrix of $F$,
i.e., $\sigma(A)\in Her_m^{>0}(\mathbb{R})$ for every
embedding $\sigma:F\hookrightarrow\mathbb{C}$.
$A[x]=Tr_{F/\mathbb{Q}}(A,q(x))$ always takes values in 
$\mathbb{Q}$. If we write $x=(x_1,\dots,x_m)^T$
where 
\begin{equation}
  x_i=\sum_{j=1}^{2r}x_{i,j}a_j, x_{i,j}\in\mathbb{Z},
\end{equation}
then $A[x]$ can be viewed as
a $2mr_2$-ary rational quadratic form, lying in a subspace
$T$ of $Sym_{2mr_2}(\mathbb{R})$ with dimension
$m^2r_2$. Therefore the calculation of $T$-perfect forms 
also terminates in finitely many steps.

\subsection{The polyhedral cones and cell complex}

Given integer $m>0$ and a totally real or CM field 
$F$, we now have a complete list of 
$GL_m(O_F)$-equivalence classes 
of $m$-ary perfect forms with minimum $1$
\begin{equation}
  \{A_1,A_2,\dots,A_n\}
\end{equation}
along with their minimum point sets
\begin{equation}
  \{\mathrm{Min}(A_1),\mathrm{Min}(A_2),\dots,\mathrm{Min}(A_n)\}.
\end{equation}
Define $PC_i$, the perfect cone attached to $A_i$,
as follows:
\begin{equation}
  PC_i:=\{\sum_{v\in\mathrm{Min}(A_i)}\lambda_vq(v)|\lambda_v>0\}
\end{equation}
where $q(v)$ is regarded the same as
its image under the infinite places:
$(\sigma(q(v)))_{\sigma:F\hookrightarrow\mathbb{C}}$.
Then
for every arithmetic subgroup $\Gamma$ of $GL_2({O}_F)$,
$\{PC_i\}_{i=1}^n$ form a $\Gamma$-admissible decomposition of 
$(Her_m^{>0}(\mathbb{C}))^{r}$ (resp. $(Sym_m^{>0}(\mathbb{R}))^{r_1}$)
if $F$ is CM (resp. totally real).

Modulo $(\mathbb{R}_+)^*$,
when $m=2$ and $F$ is a CM field, 
the cones descend to a $PGL_2(O_F)$-tessellation of 
$(\mathbb{H}_3)^{r}$.
We have a complete list of ideal faces of
dimensions from $0$ to $3r$.
The $0$-faces are ideal vertices
of $P_F^1=(\sigma_1,\dots,\sigma_r)(P_F^1)\hookrightarrow
P_\mathbb{C}^1$.

\section{Torsion-free part of Bloch groups}

In previous chapters, we give a detailed process
to find a $PGL_2(O_F)$-tessellation of $(\mathbb{H}_3)^{r}$ by ideal polytopes.
This chapter will give an algorithm
to find a full-rank subgroup of $\overline{B}(F)_{tf}$,
and determine its index.

\subsection{Cross-ratios, Bloch groups and $K$-groups}

We first recall some connections between
Bloch groups and $K$-theory.

For $m>1$ an integer, and a number field $F$,
Quillen's $K$-group $K_m(O_F)$ is 
the $m$-th homotopy group of an $H$-space $BGL(F)^+$ having the same homology as 
$\displaystyle\lim_{\longrightarrow}GL_n(F)$. 
By Borel's theorem, if $m>1$, the rank of $K_m(F)$
is $0$ if $m$ is even, $r_2$ if $m\equiv3\pmod4$,
and $r_1+r_2$ if $m\equiv1\pmod4$.

Now we introduce definitions and conclusions
about Bloch group. 

(1)Let $F$ be a CM field.
Let $C_n(P_F^1)$ be the abelian group generated by $(P_F^1)^{n+1}$.
The map $d_n:C_n(P_F^1)\rightarrow C_{n-1}(P_F^1)$ is defined as
\begin{equation}
  d_n(l_0,\dots,l_n)=\displaystyle\sum_{i=0}^n (-1)^{i-1}(l_0,\dots,\hat{l_i},\dots,l_n),
\end{equation}
where the hat means this entry is deleted.
Then we have a complex $(C_*(P_F^1),d_*)$. By Lemma 2.11 of \cite{MR1348706},
it is acyclic. Therefore, if we denote by $C_n(\mathcal{L})$ the 
abelian group generated by $GL_2(F)\backslash(P_F^1)^{n+1}$,
then the induced complex
$(C_*(\mathcal{L}),d)$, computes the homology of $GL_2(F)$.

Define a homomorphism
\begin{equation}
  f_{2,F}: C_2(P_F^1)\rightarrow{\wedge}^2(F^*)
\end{equation}
by sending non-degenerate $(l_0,l_1,l_2)$ to 
\begin{equation}
  cr(l_0,l_1,l_2):=\det(l_0,l_1)\wedge\det(l_0,l_2)-
  \det(l_0,l_1)\wedge\det(l_1,l_2)+
  \det(l_0,l_2)\wedge\det(l_1,l_2).
\end{equation}
One could easily verify that it is independent of the lengths of $l_i$,
and for $g\in GL_2(F)$,
\begin{equation}
  cr_2(gl_0,gl_1,gl_2)=cr_2(l_0,l_1,l_2).
\end{equation}
Therefore we have the map $f_{2,F}:C_*(\mathcal{L})\rightarrow{\wedge}^2(F^*).$

Define another homomorphism 
\begin{equation}
  f_{3,F}: C_3(P_F^1)\rightarrow\mathbb{Z}[F^\flat]
\end{equation}
by sending degenerate tuples to zero and non-degenerate $(l_0,l_1,l_2,l_3)$ to 
$[cr_3(l_0,l_1,l_2,l_3)]$,
which is well-defined, and has the following properties:

\begin{enumerate}[(a)]
  \item $cr_3(gl_0,gl_1,gl_2,gl_3)=cr_3(l_0,l_1,l_2,l_3)$ for every $g$ in $GL_2(F)$.
  \item $cr_3((0,1)^T,(1,0)^T,(1,1)^T,(x,1)^T)=x$ for $x\in F^\flat$.
  \item Even permutations of $l_i$ change the $cr_3=x$ into $x,1-x^{-1},(1-x)^{-1}$, while odd ones into $1-x,x^{-1},(1-x^{-1})^{-1}$.
\end{enumerate}

The map $f_{3,F}$ onto $\mathbb{Z}[F^\flat]$ is compatible with action of $S_4$.
Moreover, $f_{2,F}$ is also compatible with action of $S_3$ on oriented triangles.

Then Burns et al proved that
\begin{theorem}[Theorem 4.1 of \cite{BURNS20121502}]
  \begin{enumerate}[(a)]
    \item There is a homomorphism 
    \begin{equation}
      \phi_F: Ker(\delta_{2,F})\rightarrow K_3(F)_{tf}^{ind}
    \end{equation}
    that is natural up to sign, and functorial in $F$ after fixing a choice of sign.
    \item $Coker(\phi_F)$ is finite if $F$ is a number field.
    \item Let $reg_2$ be Beilinson's regulator map $K_3(\mathbb{C})\rightarrow\mathbb{R}$.
    There exists a universal choice of sign such that if $F$ is any number field and $\sigma:F\rightarrow\mathbb{C}$ is any embedding,
    then the composition
    \begin{equation}
      reg_\sigma: Ker(\delta_{2,F})\stackrel{\phi_F}{\longrightarrow}K_3(F)_{tf}^{ind}=K_3(F)_{tf}\stackrel{\sigma_*}{\longrightarrow}K_3(\mathbb{C})_{tf}\stackrel{reg_2}{\longrightarrow}\mathbb{R}
    \end{equation}
    is induced by mapping $[x]$ in $F^\flat$ to $iD(\sigma(x))$.
  \end{enumerate}
\end{theorem}

The set $\mathcal{L}$ has a natural map to 
the set of $F$-rational rays of $Her_2^{>0}(F)$:
\begin{equation}
  x=(x_0,x_1)\rightarrow \mathbb{R}_{>0}q(xx^*)
\end{equation}

Let $\partial_{2,F}$ be the map
\begin{equation}
  \mathbb{Z}[F^\flat]\rightarrow\wedge^2(F^*),[x]\mapsto(1-x)\wedge x.
\end{equation}
One could easily verify that 
$\partial_{2,F}$ sends the five-term relation in (\ref{five-term}),
$[x]+[1-x]$ and $[x]+[x^{-1}]$ to $0$ in $\wedge^2(F^*)$. Therefore, 
$\partial_{2,F}$ is well-defined on $\mathcal{P}(F)$.
\begin{remark}
  The relations in $\mathcal{P}(F)$ have obvious explanation in geometry.
  Let $l_0,\dots,l_4\in P_F^1\in\partial\mathbb{H}_3$, and $T_i$ the volume
  of oriented tetrahedron with ideal vertices $l_0,\dots,\hat{l_i},\dots,l_4$.
  Then $T_0-T_1+T_2-T_3+T_4=0$, which reflects the five-term relationship.
  $x$ and $1-x$ are exactly two cross-ratios of two orientations of 
  one ideal tetrahedron, and $x,\bar{x}$ are cross-ratios of two tetrahedra
  symmetric respect to the half-plane $(\mathbb{R},0)\times\mathbb{R}^+\subset\mathbb{H}_3$.
\end{remark}

Finally, we have the commutative diagram below:

\begin{equation}\label{comdiag}
  \begin{tikzcd}
    \dots \arrow[r, "d_5"] & C_4(\mathcal{L}) \arrow[r, "d_4"] \arrow[d] & C_3(\mathcal{L}) \arrow[r, "d_3"] \arrow[d, "f_{3,F}"] &C_2(\mathcal{L}) \arrow[r, "d_2"] \arrow[d,"f_{2,F}"] 
    & C_1(\mathcal{L}) \arrow[r, "d_1"] \arrow[d] & C_0(\mathcal{L}) \\
    \  & 0 \arrow[r] & \mathcal{P}(F) \arrow[r,"\partial_{2,F}"] & \wedge^2F^* \arrow[r] & 0 & \ 
\end{tikzcd}
\end{equation}
And therefore, we have an induced surjective map
$$H_3(C_*(\mathcal{L}))\rightarrow B(F)$$
by sending an oriented simplex with vertices in $P_F^1$ to its $[cr_3]$.

\subsection{Orbifold homology}

In Chap. 3 we give an algorithm to calculate equivalence classes
of binary perfect forms, which descends to 
$GL_2(O_F)$-tessellation of $\mathbb{H}_3$. 
We have a complete list of ideal faces with vertices
in the embedding
of $P_F^1$ into $(P_\mathbb{C}^1)^{r}$.
Let $\Sigma_{s}$ the set of the equivalence classes of $GL_2(O_F)$-equivalent
classes of $s$-dimension faces, $s=0,1,\dots,3r$.
Then there is a differencial map 
\begin{equation}
  d_s:\Sigma_{s}\rightarrow\Sigma_{s-1}
\end{equation}
by sending a simplex $[A_0\dots A_s]$ to
$\displaystyle\sum_{i=0}^{s}(-1)^{i-1}[A_0\dots\hat{A_i}\dots A_s]$.

(Note that if a polyhedron is not a simplex, then it can be subdivided to 
sums of some simplices with their internal faces glued, and vanishes
under $d_*$.
For example, a general planar quadrilateral
$ABCD$ can be subdivided to $ABD+BCD$,
such that the common diagonal $BD$ of the two triangles is 
eliminated. This operation is called triangulation.)

Then $(\Sigma_*,d_*)$ is a complex. 
Note that $\Sigma_{s}$ is a finite set.

\begin{theorem}
  The rank of $H_3(\Sigma_{*},d_*)$ is $r$.
\end{theorem}

\begin{proof}
  Denote $M$ the orbifold $PGL_2(O_F)\backslash(\mathbb{H}_3)^{r}$,
  $W$ its Borel-Serre compactification, i.e.,  with ideal vertices added,
  and $\partial W=W-M$.

  We define another complex 
  on $\Sigma_*$, which will be more useful in the next section.
  Let $N$ be the least common multiple of orders 
  of automorphism groups (in $PGL_2(O_F)$)
  of all faces. For every face $\mathfrak{a}$,
  define 
  \begin{equation}
    w(\mathfrak{a})=\frac{N}{\mathrm{Aut}(\mathfrak{a})}.
  \end{equation}
  Then for every face $\mathfrak{a}'$ of $\mathfrak{a}$,
  we have $w(\mathfrak{a})|w(\mathfrak{a}')$. Define 
  \begin{equation}
    \partial_s:\Sigma_s\rightarrow\Sigma_{s-1},\mathfrak{a}\mapsto\displaystyle\sum_{j=0}^n\frac{w(\partial_j(\mathfrak{a}))}{w(\mathfrak{a})}(-1)^j\partial_j(\mathfrak{a}).
  \end{equation}
  Then $(\Sigma_*,\partial_*)$ is a descending weighted simplicial complex.
  (For definition see Section 2.3 of \cite{wei2025weightedhomologytheoryorbifolds}.)
  
  (For example, let $A-BCD$ be a regular triangular pyramid.
  It has four $2$-faces. Then 
  \begin{equation}
    \frac{w(BCD)}{w(ABCD)}=1,w(ACD)=w(ABD)=w(ABC)=3w(ABCD).
  \end{equation}
  Hence we have 
  \begin{equation}
    \partial^w(ABCD)=BCD-3ACD+3ABD-3ABC.)
  \end{equation}

  By Theorem 7.1 of \cite{wei2025weightedhomologytheoryorbifolds},
  the rational homology groups of $(\Sigma_*,d_*)$ and 
  $(\Sigma_*,\partial_*)$ are isomorphic.

  For $a\in\partial W$, denote $a=(a_1,\dots,a_{r})$,
  where $a_i=\sigma_i(a_1)$.
  Every inclusion $\sigma_i:F\hookrightarrow\mathbb{C}$ induces a $3$-cohomology class 
  in $H^{3}(\partial W;\mathbb{R})$
  \begin{equation}
    V_i(a,b,c,d)=D(cr_3(a_i,b_i,c_i,d_i))=D(\sigma_i(cr_3(a_1,b_1,c_1,d_1))),
  \end{equation}
  We have the exact sequence
  \begin{equation}
    0=H^{3r}(W)\rightarrow H^{3r}(\partial W)\rightarrow H^{3r}(W,\partial W)\rightarrow0.
  \end{equation}
  So by Lefschetz duality, we have 
  \begin{equation}
    H^{3r}(\partial W;\mathbb{Q})\simeq H^{3r}(W,\partial W;\mathbb{Q})\simeq H_0(W;\mathbb{Q})
  \end{equation}
  has rank $1$.
  Then $H^{3r}(\partial W;\mathbb{Q})$ is generated by $\wedge_{i=1}^n V_i\in\wedge_{i=1}^n H_3(\partial W;\mathbb{Q})$.
  By duality, $H_3(\partial W;\mathbb{Q})\simeq H^3(\partial W;\mathbb{Q})$
  has rank $r$.

  Therefore, both $(\Sigma_*,d_*)$ and $(\Sigma_*,\partial_*)$ have 
  $\mathrm{rank}H_3=r$. 
  
  The proof is complete.

\end{proof}

\subsection{Volume of orbifold and zeta function at $2$}

Let $t_1,\dots,t_{r}$ be a 
$\mathbb{Q}$-basis of the descending weighted simplicial complex
$H_3(\Sigma_*,\partial_*;\mathbb{Q})$ calculated in the previous part.
Without loss of generality, assume $t_i$s are of the following forms:
\begin{equation}
  t_i=\displaystyle\sum_{j=1}^{n_i}\frac{1}{|\Gamma_{i,j}|}[T_{i,j}],
\end{equation}
(otherwise use a positive integer to divide them)
where $\Gamma_{i,j}$ is the automorphism group of the oriented 3-simplex 
\begin{equation}
  T_{i,j}=[a_{i,j},b_{i,j},c_{i,j},d_{i,j}].
\end{equation}
where $a_{i,j}=(\sigma_1(a_{i,j}),\dots,\sigma_{r}(a_{i,j}))\in Im(P_F^1\hookrightarrow P_\mathbb{C}^1)$,
and so forth.

\begin{theorem}
  Define a matrix $Y_{i,k=1}^{r}$ as follows:
  \begin{equation}
    Y_{i,k}=D(cr_3(\displaystyle\sum_{j=1}^{n_i}\frac{1}{|\Gamma_{i,j}|}\sigma_k(a_{i,j},b_{i,j},c_{i,j},d_{i,j}))).
  \end{equation}
  Then we have 
  \begin{equation}
    |\det(Y)|=vol(\Gamma\backslash(G')^{r}).
  \end{equation}
\end{theorem}

\begin{proof}
  We regard $t_i$ as third homology classes in 
  $H_3(\partial W,\mathbb{Q})$, and 
  $\omega_k=D_k\circ cr_3$ the cohomology classes.
  Then we have the evaluation pairing
  \begin{equation}
    Y_{i,k}=\int_{t_i}\omega_k=\langle t_i,\omega_k\rangle.
  \end{equation}
  Then we have the formula
  \begin{equation}
    \langle t_1\cup\dots\cup t_{r}, \omega_1\wedge\dots\wedge\omega_{r}\rangle
    =\det(Y).
  \end{equation}
  
  More precisely,
  since the degree of components $3$ is odd, both
  the cup products and wedge products are anticommutative.
  Therefore for $\phi\in S_{r}$, the 
  component pairing $t_{\phi(1)},\dots,t_{\phi(r)}$ with 
  $\omega_1,\dots,\omega_{r}$, is $(-1)^\phi
  \langle t_{\phi(1)},\omega_1 \rangle\dots\langle t_{\phi(r)},\omega_r \rangle$,
  where $(-1)^\phi$ is the sign of $\phi$, reflecting whether the orientation
  of $t_{\phi(1)}\cup\dots\cup t_{\phi(r)}$ and 
  $t_{1}\cup\dots\cup t_{n}$ is the same or different.

  The cohomology classes are $PGL_2(\mathbb{C})$-equivalent,
  and what we have found is a $PGL_2(O_F)$-basis. 
  By the definition of the covolume of group lattice,
  the proof is complete.

\end{proof}

\begin{corollary}
  \begin{equation}\label{detvalue}
    |\det(Y)|=2^{1-4r}\pi^{-2r}|D_F|^{3/2}\zeta_F(2).
  \end{equation}
\end{corollary}

\begin{proof}
  Consider the surjective homomorphism
  \begin{equation}
    PGL_2(O_F)\rightarrow O_F^*/(O_F^*)^2,g\mapsto\det(g).
  \end{equation}
  The kernel is exactly $PSL_2(O_F)$, and right hand side is 
  isomorphic to $(\mathbb{Z}/2)^{r}$.
  Therefore every $PGL_2(O_F)$-class 
  contains $2^r$ $PSL_2(O_F)$-classes.
  Hence the covolume of $PSL_2(O_F)$ is $2^r$
  times the covolume of $PGL_2(O_F)$.

  Now apply 7.4 of \cite{MR616899}:
  \begin{equation}
    covol(PSL_2(O_F))=2^{1-3r}\pi^{-2r}|D_F|^{3/2}\zeta_F(2).
  \end{equation}
  The proof is complete.
\end{proof}

\begin{corollary}
  If $F$ is an abelian CM field, then
  the subgroup generated by $t_1,\dots,t_r$ has index $\frac{|K_2(O_F)|}{|K_3(O_F)_{tor}|}$
  in $B(F)\otimes\mathbb{Q}$.
\end{corollary}

\begin{proof}
  Combining (\ref{detvalue}) with Theorem \ref{BK1} and 
  Proposition \ref{BK2},
  and the following function equation for $\zeta_F(s)$ at $2$ and $-1$:
  \begin{equation}
    \zeta_F^*(-1)=(-1)^{r}\frac{|D_F|^{3/2}}{2^{2r-1}\pi^{2r}}\zeta_F(2).
  \end{equation}
\end{proof}

\begin{remark}
  Let $|N_F|$ be the least common multiple of 
  indexes of all possible finite subgroups of $PGL_2(O_F)$.
  In fact by scaling $t_i$ by integers dividing $|\mu_F||N_F|$, 
  one can easily find a subgroup of finite index of $B(F)$.
  To find a basis of $B(F)$ itself, one could apply 
  the method introduced in Chap. 6 of \cite{MR4264211}.
  The calculation involves solving $S$-unit equations,
  which takes much time if $F$ has degree $\geq4$, and $S$ is large.
\end{remark}

\section*{Aknowledgements}

The author would like to express sincere gratitude to 
his supervisor, Prof Hourong Qin, for his guidance and suggesting this problem.
Also, the author thanks to support by NSFC (Nos.12231009, 11971224).

\bibliographystyle{unsrt}
\bibliography{114514}

\end{document}